\documentclass[reqno,12pt]{article}

\usepackage[utf8]{inputenc}
\usepackage{amsbsy}
\usepackage{amsfonts}
\usepackage{dsfont}
\usepackage{amsmath}
\usepackage{amsthm}
\usepackage{amssymb}
\usepackage{enumerate}
\usepackage{hyperref}
\usepackage{color}
\usepackage{etoolbox}
\usepackage[left=3cm,top=2cm,bottom=3cm,right=2.5cm]{geometry}

\title{The limiting characteristic polynomial of classical random matrix ensembles}

\author{Reda   \textsc{Chhaibi}      \footnote{\texttt{reda.chhaibi@math.univ-toulouse.fr}},
        Emma   \textsc{Hovhannisyan} \footnote{\texttt{emma.hovhannisyan@math.uzh.ch}},
        Joseph \textsc{Najnudel}     \footnote{\texttt{joseph.najnudel@bristol.ac.uk}},\\
        Ashkan \textsc{Nikeghbali}   \footnote{\texttt{ashkan.nikeghbali@math.uzh.ch}},
        Brad   \textsc{Rodgers}      \footnote{\texttt{brad.rodgers@queensu.ca}}
       } 
\allowdisplaybreaks[4]

\theoremstyle{plain}

\newtheorem{thm}{Theorem}[section]
\newtheorem{proposition}[thm]{Proposition}
\newtheorem{corollary}[thm]{Corollary}

\newtheorem{definition}[thm]{Definition}

\newtheorem{lemma}[thm]{Lemma}

\numberwithin{equation}{section}

\def\half{\frac{1}{2}}

\def\a{\alpha}
\def\b{\beta}

\def\N{{\mathbb N}}

\def\R{{\mathbb R}}
\def\C{{\mathbb C}}

\def\P{{\mathbb P}}
\def\E{{\mathbb E}}

\newcommand\Var{\mathrm{Var}}

\begin{document}

\maketitle

\begin{abstract}
We demonstrate the convergence of the characteristic polynomial of several random
matrix ensembles to a limiting universal function, at the microscopic scale. The random matrix ensembles we treat are classical compact groups and the Gaussian Unitary Ensemble.
	
In fact, the result is the by-product of a general limit theorem for the convergence of random entire functions whose zeros present a simple regularity property.
\end{abstract}

\indent
\hrule
\tableofcontents
\indent
\hrule

\section{Introduction}
\label{section:introduction}

A novel perspective in random matrix theory takes as its fundamental object a random function, such as the characteristic polynomial of a random matrix, and has as a goal demonstrating the convergence of these random functions  to a limiting object. This is in contrast to a more traditional focus on demonstrating the convergence of a point process such as the collection of eigenvalues to some random limit. Much of the information that has traditionally been of interest can be succinctly summarized within this new framework and furthermore new questions are brought to the surface by the change of focus. Papers making use of this perspective include \cite{AiWa,ChNaNi,So}.

In the recent paper \cite{ChNaNi} the first, third, and fourth authors of this paper introduced the following random entire function:
\begin{equation}
\label{limit_char}
\xi_\infty(s):= \lim_{B\rightarrow\infty} e^{i\pi s} \prod_{|y_i| \leq B} \Big(1-\frac{s}{y_i}\Big),
\end{equation}
where $y_i$ are the points of the determinantal point process with sine kernel, also called {\it Sine process}. Such a product, it was shown, will converge for all $s\in \mathbb{C}$ almost surely and gives a random entire function. We call it the \emph{limiting characteristic polynomial} for the following reason: if $U(n)$ is the group of $n\times n$ unitary matrices, endowed with Haar measure, and $g$ is a random element of $U(N)$, then for
$$
\xi_n(s) := \frac{\det(e^{i2\pi s/n}-g)}{\det(1-g)},
$$
it was shown in \cite{ChNaNi} that the random analytic function $\xi_n$ tends in distribution to $\xi_\infty$ in the topology of uniform convergence on compact sets. The proof in that paper proceeds from the machinery of virtual isometries, special to the unitary group developed in \cite{BNN}. Such a result is closely related to, though does not follow only from, the fact that the rescaled eigenangles of a random unitary matrix tend locally to a Sine process. 

In this paper we consider a generalization of this result for other point processes. We first prove a general result for a broad class of point processes, and then apply this result to several random matrix ensembles. A thorough introduction to point processes may be found in \cite{Kall}.

\begin{definition}
\label{def:amenable}
Let $x^{(1)}$, $x^{(2)}$, $x^{(3)}$,... all be point processes with points in $\mathbb{R}$, and denote the point instances of the point process $x^{(n)}$ by $\{x^{(n)}_i\}_{i\in\mathbb{Z}}$. Define the random variables
$$
X_I^{(n)}:= \#\{ x_i^{(n)} \in I\},
$$
for intervals $I$, which count the number of points of $x^{(n)}$ that lie in $I$. We say that the sequence of point processes $x^{(1)}, x^{(2)},...$ is \textbf{uniformly product-amenable} if, for all $n$ the sums
\begin{equation}
\label{converging_sums}
\lim_{B\rightarrow\infty} \sum_{|x_i^{(n)}| \leq B} \frac{1}{x_i^{(n)}}, \quad \textrm{and} \quad \sum \frac{1}{|x_i^{(n)}|^2}
\end{equation}
converge almost surely, and if for some constant $\delta > 0$, the following estimates are true uniformly in $n$ for all intervals $I$ with $|I| \geq 1$, 
\begin{enumerate}[(i)]
\item Symmetry in expectation: $\mathbb{E} X_I^{(n)} = \mathbb{E} X_{-I}^{(n)} + O(|I|^{1-\delta}).$
\item Regularity in expectation:  $\mathbb{E} X_I^{(n)} \lesssim |I|.$
\item Regularity in variance: $\Var\; X_I^{(n)} \lesssim |I|^{2-\delta}.$
\end{enumerate}
\end{definition}
Note the notation $f_n \lesssim g_n$ or $f_n = O\left(g_n \right)$ means that there exists $C>0$ such that $|f_n| \leq C g_n$ for all $n$.

Our first theorem is as follows:

\begin{thm}
\label{general_pp_theorem}
Assume that $x^{(1)}, x^{(2)}, ...$ are a sequence of uniformly product-amenable point processes. Define
$$
c_n(s) := \lim_{B\rightarrow\infty} e^{i\pi s} \prod_{|x_i^{(n)}| \leq B} \Big(1-\frac{s}{x_i^{(n)}}\Big),
$$
This product converges almost surely for all $n$ and all $s\in \mathbb{C}$. Moreover, if in law the point processes $x^{(n)}$ tend to the Sine process $y$, then in law
$$
c_n(s) \rightarrow \xi_\infty(s),
$$
in the topology of uniform convergence on compact sets.
\end{thm}

This theorem applies to a wide variety of point processes that are of interest in random matrix theory. As a demonstration, we apply it to random matrix ensembles generated by the classical compact groups and the Gaussian Unitary Ensemble (GUE).

\begin{thm}[Random orthogonal matrices]
\label{thm:ortho}
Denote the group of $n\times n$ special orthogonal matrices by $SO(n)$, and let $g$ be a random element of $SO(n)$ chosen according to Haar measure. For any fixed nonzero $E \in (-\frac{1}{2},\frac{1}{2})$, define the random function
$$
\xi_n^{SO}(s):= \frac{\det(e^{i2\pi(E+s/n)} - g)}{\det(e^{i 2\pi E} - g)}.
$$
Then in the topology of uniform convergence on compact sets,
$$
\xi_n^{SO}(s) \rightarrow \xi_\infty(s)
$$
in law.
\end{thm}

Recall that the group of special orthogonal matrices is defined as follows: $SO(n) := \{g \in \mathcal{M}_{n\times n}(\mathbb{R}): g g^{T} = I,\; \det(g) = 1\}$. Here $\mathcal{M}_{n\times n}(\mathbb{R})$ is the set of all $n\times n$ matrices with real entries, and $g^t$ is the transpose of $g$.

\begin{thm}[Random symplectic matrices]
\label{thm:symp}
For $n$ an even integer, denote the group of $n\times n$ symplectic matrices by $Sp(n)$, and let $g$ be a random element of $Sp(n)$ chosen according to Haar measure. For any fixed nonzero $E \in (-\frac{1}{2},\frac{1}{2})$, define the random function
$$
\xi_n^{Sp}(s):= \frac{\det(e^{i2\pi(E+s/n)} - g)}{\det(e^{i 2\pi E} - g)}.
$$
Then in the topology of uniform convergence on compact sets,
$$
\xi_n^{Sp}(s) \rightarrow \xi_\infty(s)
$$
in law.
\end{thm}

Recall that the group of symplectic matrices is defined \emph{only for even} $n$ as follows: $Sp(n) := \{ g \in U(n):\; g J g^t = J\}$ where 
$$
J := \begin{pmatrix} 0 & I_{n/2} \\ - I_{n/2} & 0 \end{pmatrix}.
$$
The following corollary is the direct application of Theorems \ref{thm:ortho} and \ref{thm:symp}. In what follows, $g$ represents a random element of $SO(n)$ or $Sp(n),$ chosen according to Haar measure. 
\begin{corollary}
Let $r \in \N$ and $\a_j , \b_j \in \C$ for all $1 \leq j \leq r.$ Then as $n \to \infty$
\begin{align}
R\left( \a_1, \dots, \a_r; \, \b_1, \dots, \b_r\right) &:= \prod_{i=1}^r \frac{\det(e^{i2\pi(E+\a_i/n)}- g) }{\det(e^{i2\pi(E+\b_i/n)}- g)}  \nonumber\\
&\to \frac{\xi_\infty(\a_1) \cdots \xi_\infty(\a_r)}{\xi_\infty(\b_1) \cdots \xi_\infty(\b_r)}
\end{align}
in the topology of uniform convergence on compact sets, in law.
\end{corollary}
Both $SO(n)$ and $Sp(n)$ are subgroups of the unitary group $U(n)$, and so all elements from either group have $n$ eigenvalues, all located on the unit circle.

We prove a theorem similar to Theorems \ref{thm:ortho} and \ref{thm:symp} for the Gaussian Unitary Ensemble (GUE). A $n\times n$ matrix $M$ from this ensemble is defined in the following way: we set 

$$
M = \frac{1}{\sqrt{2}}(B+B^\ast),
$$
for $B$ an $n\times n$ random matrix with i.i.d. random variable entries $Z_{ij}$ each with the distribution $N_\mathbb{C}(0,1)$ of a complex normal random variable with mean $0$ and variance $1$. $M$ is by definition Hermitian, with $n$ real eigenvalues, almost surely all contained in an interval of the form $(-(2+o(1))\sqrt{n}, (2+o(1))\sqrt{n})$ as $n\rightarrow\infty$. We show the following:

\begin{thm} 
\label{thm:main}
Fix $E\in (-2,2)$ and let $M$ be a $n\times n$ GUE matrix. Define the random analytic function,
\begin{align}
\label{eq:def_Xi_n}
\Xi_n^{GUE}(s) := & \frac{\det\Big(-\frac{s}{\rho_{sc}(E)\sqrt{n}} - E\sqrt{n}+M\Big)}{\det(-E\sqrt{n}+M)}
\end{align}
with 
\begin{align}\label{sem_law}
\rho_{sc}(x):= \frac{1}{2 \pi}\sqrt{(4-x^2)_+}.
\end{align}
Then in the topology of uniform convergence on compact sets,
\begin{align}
\label{eq:def_Xi_infty}
\Xi_n^{GUE}(s) \rightarrow e^{s(E/2\rho_{sc}(E)-i\pi)}\xi_\infty(s)
\end{align}
in law.
\end{thm}
\begin{corollary}
Let $r \in \N$ and $\a_j , \b_j \in \C$ for all $1 \leq j \leq r.$ Then as $n \to \infty$
\begin{align}
R^{GUE}\left( \a_1, \dots, \a_r; \, \b_1, \dots, \b_r\right) 
&:= \prod_{i=1}^r\frac{\det\Big(-\frac{\a_i}{\rho_{sc}(E)\sqrt{n}} - E\sqrt{n}+M\Big) }{\det\Big(-\frac{\b_i}{\rho_{sc}(E)\sqrt{n}} - E\sqrt{n}+M\Big) }  \nonumber\\
&\to e^{(E/2\rho_{sc}(E)-i\pi) \sum_{i=1}^r \left(\a_i-\b_i \right)}\frac{\xi_\infty(\a_1) \cdots \xi_\infty(\a_r)}{\xi_\infty(\b_1) \cdots \xi_\infty(\b_r)}
\end{align}
in the topology of uniform convergence on compact sets, in law.
\end{corollary}
Sodin \cite{So} notes that Theorem \ref{thm:main} follows by a careful integration from a theorem proved by Aizenman and Warzel \cite{AiWa} for the logarithmic derivative of these random entire functions. Aizenman and Warzel's proof relies on the theory surrounding the Nevanlinna class of analytic functions. The alternative proof we give here may be of independent interest, relying, as it does, less on complex function theory. In particular in Proposition \ref{proposition:localization1} we establish a localization of the function $\Xi_n^{GUE}(s)$, and in Proposition \ref{proposition:point_count_estimates} we establish uniform estimates for counts of eigenvalues of GUE matrices, both of which we believe do not previously appear in the literature.

It is worth noting that there is nothing really special about the Sine process in the arguments that follow, except that the characteristic polynomials we will be interested in (among others) converge to it, and that the limit object $\xi_\infty(s)$ defined by \eqref{limit_char} is already known to converge almost surely (this is proved in \cite{ChNaNi}).

\section{From point processes to random functions: a proof of Theorem \ref{general_pp_theorem}}

Our proof of Theorem \ref{general_pp_theorem} proceeds along the following outline: We show that 1) in probability the functions $c_n(s)$ and $\xi_\infty(s)$ can be approximated on a compact set by truncated products over elements of the point process, and then 2) that these truncated products converge to one another as $n\rightarrow\infty$ owing to the convergence of the zero processes.

\subsection{Localization of products}

We begin with
\begin{lemma}
\label{product_truncation1}
For a sequence of uniformly product-amenable point processes $x^{(1)}, x^{(2)}, ...$ with the sums in \eqref{converging_sums} converging almost surely, we have for any compact set $K \subset \mathbb{C}$ and $\varepsilon > 0$,
$$
\lim_{A\rightarrow\infty} \sup_n \;\mathbb{P}\Big(\sup_{s\in K} \Big| \prod_{\left|x_i^{(n)}\right| \geq A} \Big(1 - \frac{s}{x_i^{(n)}}\Big) - 1 \Big| \geq \varepsilon \Big) = 0.
$$
\end{lemma}

\begin{proof}
For $A$ sufficiently large (based on $K$), we have
\begin{equation}
\label{exp_expand}
\prod_{|y_i^{(n)}| \geq A} \left( 1 - \frac{s}{x^{(n)}_i}\right) = \exp\left(-s \sum_{|x_i^{(n)}| \geq A} \frac{1}{x_i^{(n)}} + O\left(\sum_{|x_i^{(n)}| \geq A} \frac{1}{|x_i^{(n)}|^2}\right)\right).
\end{equation}
(That the product on the left hand side converges almost surely is clear from this relation.) We show that as $A\rightarrow\infty$,
\begin{equation}
\label{sum_1/y}
\sup_n \,\mathbb{P}\Big(\Big| \sum_{\left|x_i^{(n)}\right| \geq A} \frac{1}{x_i^{(n)}}\Big| \geq \varepsilon \Big) \rightarrow 0,
\end{equation}
\begin{equation}
\label{sum_1/y^2}
\sup_n \,\mathbb{P}\Big( \sum_{|x_i^{(n)}| \geq A} \frac{1}{ |x_i^{(n)}|^2} \geq \varepsilon \Big) \rightarrow 0,
\end{equation}
which from \eqref{exp_expand} are sufficient to prove the lemma.

We examine \eqref{sum_1/y^2} first. Note that from a dyadic decomposition,
$$
\mathbb{E}\; \sum_{|x_i^{(n)}| \geq A} \frac{1}{|x_i^{(n)}|^2} \leq \mathbb{E} \sum_{2^{k+1} \geq A} \frac{1}{4^k}\left(X_{[2^k, 2^{k+1})} + X_{-[2^k, 2^{k+1})}\right) \lesssim \sum_{2^{k+1} \geq A} \frac{1}{2^k},
$$
by the regularity in expectation of the counts $X_I$ (Definition \ref{def:amenable}, Condition \emph{(ii)}). As $A\rightarrow\infty$ this tends to $0$ uniformly in $n$, and \eqref{sum_1/y^2} follows from Markov's inequality.

For \eqref{sum_1/y}, rather than a dyadic decomposition, we make a decomposition along intervals $[\ell^3, (\ell+1)^3),$
\begin{align*}
\sum_{|y_i^{(n)}| \geq A} \frac{1}{y_i^{(n)}} &= \sum_{(\ell+1)^3 \geq A} \sum_{\substack{ |y_i^{(n)}| \geq A \\ \ell^3 \leq |y_i^{(n)}| < (\ell+1)^3}} \frac{1}{y_i^{(n)}} \\
&= \sum_{(\ell+1)^3 \geq A} \frac{X_{[\ell^3, (\ell+1)^3)} - X_{-[\ell^3, (\ell+1)^3)}}{\ell^3}
\\ &  + O\Big( \sum_{(\ell+1)^3 \geq A} \frac{1}{\ell^4} (X_{[\ell^3, (\ell+1)^3)} + X_{-[\ell^3, (\ell+1)^3)} )\Big).
\end{align*}
Writing $Z_\ell:= X_{[\ell^3, (\ell+1)^3)} - X_{-[\ell^3, (\ell+1)^3)}$ for notational reasons, symmetry in expectation (Definition \ref{def:amenable}, Condition \emph{(i)}) implies that
\begin{align*}
\Big\| \sum_{|y_i^{(n)}| \geq A} \frac{1}{y_i^{(n)}} \Big\|_{L^2}  & \leq \Big\| \sum_{(\ell+1)^3 \geq A} \frac{Z_\ell - \mathbb{E}\, Z_\ell} {\ell^3} \Big\|_{L^2} + O\Big( \sum_{(\ell+1)^3 \geq A} \frac{\ell^{2(1-\delta)}}{\ell^3}\Big),
\\ & + O\Big(  \sum_{(\ell+1)^3 \geq A}  \frac{1}{\ell^4} (  ||X_{[\ell^3, (\ell+1)^3)}||_{L^2} + 
  ||X_{-[\ell^3, (\ell+1)^3)} ||_{L^2} ) \Big)
\end{align*}
since by \emph{(i)}, $\mathbb{E}\, Z_\ell \lesssim |[\ell^3, (\ell+1)^3)|^{1-\delta} \lesssim \ell^{2(1-\delta)}.$ But the above expression may be further simplified using regularity in expectation and variance.
Indeed, combining \emph{(ii)} and \emph{(iii)}, we get 
$$||X_{[\ell^3, (\ell+1)^3)}||^2_{L^2}
=  (\mathbb{E} [ X_{[\ell^3, (\ell+1)^3)}])^2 + \operatorname{Var} (X_{[\ell^3, (\ell+1)^3)})
= O \left( \ell^4 + \ell^{4 - 2 \delta} \right),$$
and 
$$\| Z_\ell - \mathbb{E}\, Z_\ell \|_{L^2} 
= O(\ell^{2 - \delta}),$$
 which gives
\begin{align*}
\Big\| \sum_{|y_i^{(n)}| \geq A} \frac{1}{y_i^{(n)}} \Big\|_{L^2} & \lesssim \sum_{(\ell+1)^3 \geq A}
\left( \frac{\ell^{2- \delta}}{\ell^3}+ \frac{\ell^{2(1-\delta)}}{\ell^3}
+ \frac{(\ell^4 + \ell^{4-2 \delta})^{\frac{1}{2}}}{\ell^4}  \right), \\
\end{align*}
which tends to zero uniformly in $n$ as $A\rightarrow\infty$. Then, \eqref{sum_1/y}  follows from Markov's inequality, and the lemma is proved.
\end{proof}
Define the random variable $Y_I:= \#\{ y_i \in I\}$ to be the count of points of the Sine process that lie in an interval $I$. Then the variables $Y_I$ also satisfy symmetry in expectation and regularity in expectation and variance (that is, $\mathbb{E} Y_I = \mathbb{E} Y_{-I} + O(|I|^{1-\delta}),$ $\mathbb{E} Y_I \lesssim |I|,$ and $\Var\; Y_I \lesssim |I|^{2-\delta}$). The same proof therefore demonstrates that 
\begin{lemma}
\label{product_truncation2}
For $y$ the Sine point process, we have for any compact set $K \subset \mathbb{C}$ and $\varepsilon > 0$,
$$
\lim_{A\rightarrow\infty} \;\mathbb{P}\Big(\sup_{s\in K} \Big| \prod_{|y_i| \geq A} \Big(1 - \frac{s}{y_i}\Big) - 1 \Big| \geq \varepsilon \Big) = 0.
$$
\end{lemma}

\subsection{Convergence of random products: pointwise and uniformly}

We also have,
\begin{lemma}
\label{pointwise_converge}
Fix $A > 0,$ as long as the point processes $x^{(n)}$ tend in law to the Sine process $y$, we have in law in the topology of pointwise convergence
$$
\prod_{|x_i^{(n)}| < A} \Big(1 - \frac{s}{x_i^{(n)}}\Big) \rightarrow \prod_{|y| < A} \Big(1 - \frac{s}{y_i}\Big).
$$
\end{lemma}

\begin{proof} Define the random functions
\begin{equation}
\label{psi_def}
\psi_n(s):= \prod_{|x_i^{(n)}| < A} \Big(1 - \frac{s}{x_i^{(n)}}\Big), \quad \psi_\infty(s):= \prod_{|y| < A} \Big(1 - \frac{s}{y_i}\Big).
\end{equation}
We must show for any fixed $s_1,s_2,...,s_k \in \mathbb{C}$ that in law
\begin{equation}
\label{pointwise_conv1}
\left( \psi_n(s_1),\psi_n(s_2),...,\psi_n(s_k) \right) \rightarrow
\left( \psi_\infty(s_1),\psi_\infty(s_2),...,\psi_\infty(s_k) \right).
\end{equation}
To that endeavor, we will need a smoothing argument. Define for $x \in \R$, $H(x) := 1 \wedge x^+$ so that the point-wise limit $\lim_{\varepsilon \rightarrow 0} H(\varepsilon^{-1} \cdot) = \mathds{1}_{\{\cdot > 0\}}$ is the Heaviside function.

Choose $\varepsilon > 0$ and define
$$
f_{s,A}^\varepsilon(v)
          := \Big(\log\Big| 1- \frac{s}{v}\Big| + i \arg\Big(1-\frac{s}{v}\Big)\Big)
             \ H( \varepsilon^{-1} |v| )
             \ H( \varepsilon^{-1} (A-|v|) )
             \ H( \varepsilon^{-1} |v-s| ),
$$
with $\arg$ being a determination of the argument. Notice that one can pick a continuous determination in $v \in \R$ unless $s$ is real, in which case there is a jump for $v = s$. This jump is handled by the smoothing term $H( \varepsilon^{-1} |v-s| )$. In any case, for all $s \in \C$, $f_{s,A}^\varepsilon$ is continuous and compactly supported on $\R$.

By weak convergence of the point process $x^{(n)}$ to $y$, we have the convergence in law:
$$ \left( \sum_i f^\varepsilon_{s_j,A}(x_i^{(n)}); \ 1 \leq j \leq k \right)
   \stackrel{}{\rightarrow}
   \left( \sum_i f^\varepsilon_{s_j,A}(y); \ 1 \leq j \leq k \right) \ .
$$

Now define,
$$
\widetilde{\psi}_n(s)      := \exp\Big( \sum_i f^\varepsilon_{s,A}(x_i^{(n)}) \Big),
\quad 
\widetilde{\psi}_\infty(s) := \exp\Big( \sum_i f^\varepsilon_{s,A}(y_i) \Big) \ .
$$
By the continuous mapping theorem, the following convergence in law holds:
\begin{equation}
\label{tilde_converge}
\left( \widetilde\psi_n(s_1),\widetilde\psi_n(s_2),...,\widetilde\psi_n(s_k) \right)
\rightarrow
\left( \widetilde\psi_\infty(s_1),\widetilde\psi_\infty(s_2),...,\widetilde\psi_\infty(s_k) \right) \ .
\end{equation}

Yet $\widetilde\psi_\infty(s) = \psi_\infty(s)$ unless $Y_{(-\varepsilon, \varepsilon)} \geq 1$ or $Y_{(A-\varepsilon, A+\varepsilon)} \geq 1$ or $Y_{(-\varepsilon + \Re s, \Re s + \varepsilon)} \geq 1$. But
$$
\mathbb{P}( Y_{(-\varepsilon,\varepsilon)} \geq 1)
=
\mathbb{P}( Y_{(A-\varepsilon,A+\varepsilon)} \geq 1)
=
\mathbb{P}( Y_{(\Re s -\varepsilon, \Re s + \varepsilon)} \geq 1) \leq \mathbb{E}\, Y_{(-\varepsilon,\varepsilon)} = 2\varepsilon,
$$
as $y$ is the (translation invariant) Sine process. 

Because $x^{(n)} \rightarrow y$,
$$
\lim_{n\rightarrow\infty} \mathbb{P} (X^{(n)}_{(-\varepsilon, \varepsilon)} \geq 1)  \leq 2\varepsilon,
$$
and the same holds for the intervals $(A-\varepsilon,A+\varepsilon)$ and $(\Re s -\varepsilon, \Re s + \varepsilon)$. Moreover, unless $X^{(n)}_{(-\varepsilon, \varepsilon)} \geq 1$ or $X^{(n)}_{(A-\varepsilon, A+\varepsilon)} \geq 1$ or $X^{(n)}_{(-\varepsilon  + \Re s, \Re s + \varepsilon)} \geq 1$, we have $\widetilde\psi_n(s) = \psi_n(s)$. Hence with probability $1-O(\varepsilon)$, the convergence in \eqref{tilde_converge} governs that in \eqref{pointwise_conv1}. As $\varepsilon$ is arbitrary, this proves \eqref{pointwise_conv1} and thus the lemma.
\end{proof}

A simple criterion allows us to bootstrap this result to convergence in the topology of uniform convergence.

\begin{definition}
Let $\phi_1, \phi_2,...$ be a family of random functions, each mapping $\mathbb{C}\rightarrow\mathbb{C}$. We say this family is \textbf{in probability compact-equicontinuous} if for any compact $K \subset \mathbb{C}$ and for any $\varepsilon, \delta > 0$, there exists $\Delta > 0$ such that
$$
\limsup_{n\rightarrow\infty}\mathbb{P}\Big(\sup_{\substack{|s_1-s_2| < \Delta \\ s_1, s_2 \in K}} |\phi_n(s_1) - \phi_n(s_2)| \geq \delta \Big) < \varepsilon.
$$
\end{definition}

\begin{lemma}
\label{pointwise_to_uniform}
Let $\phi$ and $\phi_1, \phi_2,...$ each be a random function mapping $\mathbb{C}\rightarrow\mathbb{C}$, and suppose that in the topology of pointwise convergence, $\phi_n(s) \rightarrow \phi(s)$
in law. If the family of random functions $\phi_1, \phi_2, ...$ is in probability compact-equicontinuous, then also in the topology of uniform convergence on compact sets,
$$
\phi_n(s) \rightarrow \phi(s)
$$ 
in law.
\end{lemma}

The proof involves relatively standard ideas, but we do not know of a reference. We include a proof here for the convenience of the reader.

\begin{proof}
For $\alpha \in \mathbb{C}$, $c > 0$ and $K$ a compact set, define the events
$$
A_n(\alpha, c; K):= \{\sup_{s\in K} |\phi_n(s)-\alpha| < c\},
$$
$$
A(\alpha, c; K) := \{\sup_{s \in K} |\phi(s)-\alpha| < c \}.
$$
By the Portmanteau lemma, we will have demonstrated $\phi_n(s) \rightarrow \phi(s)$ in the topology of uniform convergence if we show that
\begin{equation}
\label{port4}
\liminf_{n\rightarrow\infty} \mathbb{P}\Big( \bigcup_{i=1}^\infty \bigcap_{j=1}^{J_i} A_n(\alpha_{ij}, c_{ij}; K_{ij})\Big) \geq \mathbb{P}\Big( \bigcup_{i=1}^\infty \bigcap_{j=1}^{J_i} A(\alpha_{ij}, c_{ij}; K_{ij})\Big),
\end{equation}
for any sequences $\alpha_{ij}, c_{ij},$ and $K_{ij}$ of complex numbers, positive numbers, and compact sets respectively. Note
\begin{equation}
\label{port3}
\liminf_{n\rightarrow\infty} \mathbb{P}\Big( \bigcup_{i=1}^\infty \bigcap_{j=1}^{J_i} A_n(\alpha_{ij}, c_{ij}; K_{ij})\Big) \geq \liminf_{n\rightarrow\infty} \mathbb{P}\Big( \bigcup_{i=1}^I \bigcap_{j=1}^{J_i} A_n(\alpha_{ij}, c_{ij}; K_{ij})\Big),
\end{equation}
for any $I$. We show that for any $I$,
\begin{equation}
\label{port2}
\liminf_{n\rightarrow\infty} \mathbb{P}\Big( \bigcup_{i=1}^I \bigcap_{j=1}^{J_i} A_n(\alpha_{ij}, c_{ij}; K_{ij})\Big) \geq \mathbb{P}\Big( \bigcup_{i=1}^I \bigcap_{j=1}^{J_i} A(\alpha_{ij}, c_{ij}; K_{ij})\Big).
\end{equation}

We begin by noting that given any $\varepsilon > 0$ and $\delta > 0$, because $\{\phi_n\}$ is in probability equicontinuous, there exist finite sets $K_{ij}^\ast \subset K_{ij}$ such that
$$
\mathbb{P}\Big( \bigcup_{i=1}^I \bigcap_{j=1}^{J_i} A_n(\alpha_{ij}, c_{ij}; K_{ij})\Big) \geq \mathbb{P}\Big( \bigcup_{i=1}^I \bigcap_{j=1}^{J_i} A_n(\alpha_{ij}, c_{ij}-\delta; K_{ij}^\ast)\Big) - \varepsilon,
$$
for all sufficiently large $n$. (In fact, $K_{ij}^\ast$ may be taken to be a $\Delta$-net of $K_{ij}$.) Since $\phi_n(s)\rightarrow \phi(s)$ in law, in the pointwise topology, again applying the Portmanteau theorem (using that $K_{ij}^\ast$ is finite for all $i,j$),
\begin{equation}
\label{port1}
\liminf_{n\rightarrow\infty} \mathbb{P}\Big( \bigcup_{i=1}^I \bigcap_{j=1}^{J_i} A_n(\alpha_{ij}, c_{ij}-\delta; K_{ij}^\ast)\Big) \geq \mathbb{P}\Big( \bigcup_{i=1}^I \bigcap_{j=1}^{J_i} A(\alpha_{ij}, c_{ij}-\delta; K_{ij}^\ast)\Big). 
\end{equation}
Now if $K' \subset K$, note that $A(\alpha, c; K') \supset A(\alpha, c; K)$. Hence the right-hand side of \eqref{port1} is
$$
\geq \mathbb{P}\Big( \bigcup_{i=1}^I \bigcap_{j=1}^{J_i} A(\alpha_{ij}, c_{ij}-\delta; K_{ij})\Big),
$$
and we have shown
$$
\liminf_{n\rightarrow\infty} \mathbb{P}\Big( \bigcup_{i=1}^I \bigcap_{j=1}^{J_i} A_n(\alpha_{ij}, c_{ij}; K_{ij})\Big) \geq \mathbb{P}\Big( \bigcup_{i=1}^I \bigcap_{j=1}^{J_i} A(\alpha_{ij}, c_{ij}-\delta; K_{ij})\Big) - \varepsilon.
$$
But $\varepsilon$ and $\delta$ are arbitrary, taking $\varepsilon, \delta \rightarrow 0$ yields \eqref{port2}. Combined with \eqref{port3} and taking $I\rightarrow\infty$, this yields \eqref{port4}.
\end{proof}

We apply this to the problem at hand by showing that the truncated products $\psi_n(s) = \prod_{|x_i^{(n)}|< A}(1-s/x_i^{(n)})$ form a family that is in probability compact-equicontinuous.

\begin{lemma}
\label{psi_continuous}
Fix $A > 0$ and define $\psi_n(s)$ as before in \eqref{psi_def}. Suppose the family of point processes $x^{(1)}, x^{(2)},...$ that define $\psi_n(s)$ is uniformly product-amenable, and $x^{(n)}\rightarrow y$, the Sine process. Then the family of random functions $\psi_1, \psi_2, ...$ is in probability compact-equicontinuous.
\end{lemma}

\begin{proof}
We need only show for any fixed compact set $K$ and constant $A > 0$,
\begin{equation}
\label{equi_proof1}
\lim_{c\rightarrow\infty} \limsup_{n\rightarrow\infty} \mathbb{P}\Big( \sup_{s\in K} \Big| \frac{d}{ds} \psi_n(s) \Big| \geq c\Big) = 0.
\end{equation}
Let $R:=\sup_{s\in K} |s|$ and note that if (for $t>0$ an arbitrarily small parameter) we have $X_{(-t,t)}^{(n)} = 0$, then by differentiating the product defining $\psi_n$,
\begin{align*}
\Big| \frac{d}{ds} \psi_n(s) \Big| =& \,\Big| \sum_{|x_j^{(n)}| < A} \frac{1}{x_j^{(n)}} \prod_{\substack{|x_i^{(n)}| < A \\ i \neq j}} \Big(1 - \frac{s}{x_i^{(n)}}\Big)\Big| \\
\leq& \,X_{(-A,A)}^{(n)}\cdot \frac{1}{t} \Big(1 + \frac{R}{t}\Big)^{X^{(n)}_{(-A,A)}}.
\end{align*}
This is a crude bound, but it suffices for our purpose. We have,
\begin{equation}
\label{crude_bound}
\mathbb{P}\Big( \sup_{s\in K} \Big| \frac{d}{ds} \psi_n(s) \Big| \geq c\Big) \leq \mathbb{P}(X_{(-t,t)}^{(n)} \geq 1) + \mathbb{P}\Big(X_{(-A,A)}^{(n)}\cdot \frac{1}{t} \Big(1 + \frac{R}{t}\Big)^{X^{(n)}_{(-A,A)}} \geq c\Big).
\end{equation}
Using regularity in expectation (Definition \ref{def:amenable}, Condition \emph{(i)}) of the point processes $x^{(n)}$ and Markov's inequality, we see
$$
\mathbb{P}\Big(X_{(-A,A)}^{(n)}\cdot \frac{1}{t} \Big(1 + \frac{R}{t}\Big)^{X^{(n)}_{(-A,A)}} \geq c\Big) \rightarrow 0,
$$
uniformly in $n$ as $c\rightarrow\infty$. But as $x^{(n)} \rightarrow y$, using Markov's inequality as in the proof of Lemma \ref{pointwise_converge}
$$
\lim_{n \rightarrow\infty} \mathbb{P}(X_{(-t,t)}^{(n)} \geq 1) \leq 2t.
$$
Using these results in \eqref{crude_bound}, we see
$$
\lim_{c\rightarrow\infty} \limsup_{n\rightarrow\infty} \mathbb{P}\Big( \sup_{s\in K} \Big| \frac{d}{ds} \psi_n(s) \Big| \geq c\Big) \leq 2t.
$$
Because $t$ is arbitrary, \eqref{equi_proof1} follows and therefore the lemma as well.
\end{proof}

\subsection{Concluding the proof of Theorem \ref{general_pp_theorem}}

From these lemmas Theorem \ref{general_pp_theorem} follows straightforwardly. For, from Lemma \ref{pointwise_converge}, Lemma \ref{pointwise_to_uniform}, and Lemma \ref{psi_continuous}, for any $A > 0$,
$$
\prod_{|x_i^{(n)}| < A} \Big(1 - \frac{s}{x_i^{(n)}}\Big) \rightarrow \prod_{|y|<A} \Big(1- \frac{s}{y_i}\Big),
$$
in law in the topology of uniform convergence on compact sets. But from Lemmas \ref{product_truncation1} and \ref{product_truncation2}, for any compact set $K$, by taking $A$ sufficiently large we can respectively approximate $c_n(s)$ and $\xi_\infty(s)$ in probability by
$$
\prod_{|x_i^{(n)}| < A} \Big(1 - \frac{s}{x_i^{(n)}}\Big), \quad \textrm{and} \quad \prod_{|y|<A} \Big(1- \frac{s}{y_i}\Big),
$$
uniformly for $s\in K$. Theorem \ref{general_pp_theorem} then follows from a standard application of the Portmanteau theorem.

\section{An application to the classical compact groups}
\label{classicalcompact}

We turn to the proofs of Theorems \ref{thm:ortho} and \ref{thm:symp} regarding the orthogonal and symplectic groups. The key facts will be the (well-known) result that eigenangles of elements of these groups away from the origin tend to the Sine process once rescaled, and the observation that the rescaled eigenangles form a uniformly product-amenable point process.

For the orthogonal group, we discuss the proof of Theorem \ref{thm:ortho} in detail. For the symplectic group, Theorem \ref{thm:symp} is very similar and its proof we will only sketch.

\subsection{Some well-known facts about the special orthogonal group}

We recall relevant facts about the special orthogonal groups. $SO(2n)$ and $SO(2n+1)$ must be treated separately; we discuss $SO(2n)$ first. (A good reference for the material recalled in this subsection is \cite{Co}.) The $2n$ eigenvalues of a matrix $g \in SO(2n)$ can be written
\begin{equation}
\label{SO(2n)_eigenangles}
\{ e^{\pm i 2\pi \vartheta_1}, ..., e^{\pm i 2\pi \vartheta_n}\},
\end{equation}
with $\vartheta_j \in [0,\frac{1}{2}]$ for all $j$. For $g$ chosen randomly according to Haar measure, the points $\{\vartheta_1,...,\vartheta_n\}$ form a determinantal point process: we have for $\eta \in C([0,\frac{1}{2}]^k)$,
\begin{equation}
\label{SO(2n)_determinantal}
\mathbb{E} \sum_{\substack{j_1, ..., j_k \\  \mathrm{distinct}}} \eta(\vartheta_{j_1}, ..., \vartheta_{j_k}) = \int_{[0,\tfrac12]^k} \eta(x_1,...,x_k) \det_{1\leq i,j \leq k} \big[K_{SO(2n)}(x_i,x_j)\big]\, d^k x,
\end{equation}
with
\begin{equation}
\label{SO(2n)_determinantal2}
K_{SO(2n)}(x,y):= s_{2n-1}(x-y) + s_{2n-1}(x+y),
\end{equation}
where
$$
s_n(x):= \frac{\sin \pi n x}{\sin \pi x}.
$$ 
The kernel $K_{SO(2n)}(x,y)$ represents an orthogonal projection in $L^2[0,\frac{1}{2}]$; that is, we have
$$
\int_0^{\frac{1}{2}} K_{SO(2n)}(x,y) K_{SO(2n)}(y,z)\, dy = K_{SO(2n)}(x,z).
$$

$SO(2n+1)$ is similar. In this case the $2n+1$ eigenvalues of $g\in SO(2n+1)$ can be written
\begin{equation}
\label{SO(2n+1)_eigenvalues}
\{1,e^{\pm i 2\pi \vartheta_1}, ..., e^{\pm i 2\pi \vartheta_n}\}
\end{equation}
with $\vartheta_j \in [0,\frac{1}{2}]$ for all $j$. For $g$ chosen randomly according to Haar measure, again the points $\{\vartheta_1,...,\vartheta_n\}$ form a determinantal point process:
\begin{equation}
\label{SO(2n+1)_determinantal}
\mathbb{E} \sum_{\substack{j_1, ..., j_k \\  \mathrm{distinct}}} \eta(\vartheta_{j_1}, ..., \vartheta_{j_k}) = \int_{[0,\tfrac12]^k} \eta(x_1,...,x_k) \det_{1\leq i,j \leq k} \big[K_{SO(2n+1)}(x_i,x_j)\big]\, d^k x,
\end{equation}
with
\begin{equation}
\label{SO(2n+1)_determinantal2}
K_{SO(2n+1)}(x,y):= s_{2n}(x-y) - s_{2n}(x+y),
\end{equation}
which again is a kernel representing an orthogonal projection on $L^2[0,\frac{1}{2}]$.

\subsection{Eigenvalue counts \texorpdfstring{for $SO(n)$}{for SO(n)} }

We first show the following:

\begin{proposition}
\label{proposition:SO(2n)_EV_bounds1}
Let $I \subset [0,\frac{1}{2}]$ be an interval and choose $g \in SO(2n)$ randomly according to Haar measure. With the first $n$ eigenangles $\vartheta_j$ labeled by \eqref{SO(2n)_eigenangles}, define the random variable $W_I^{(2n)}:= \#\{ \vartheta_j \in I\}$. Then
\begin{equation}
\label{SO(2n)_exp_bound}
\mathbb{E}\, W^{(2n)}_I = 2n|I| + O(1),
\end{equation}
and
\begin{equation}
\label{SO(2n)_var_bound}
\Var\, W^{(2n)}_I \lesssim \log(2 + 2n|I|),
\end{equation}
uniformly for all $n$ over all such intervals $I$.
\end{proposition}

\textbf{Remark:} The estimate \eqref{SO(2n)_exp_bound} is known and appears as part of Proposition 7 of \cite{MeMe}. The proof is short and so we reproduce it here. A uniform bound  of the sort \eqref{SO(2n)_var_bound} does not seem to be recorded in the literature, but closely related ideas are well known. A similar estimate for the unitary group is, for instance, proved in Lemma 7 of \cite{MeMe2}.

In our proof we make use of the following well-known computational lemma, which appears as Lemma 6 of \cite{MeMe}; the ideas behind the proof are explained in e.g. Appendix B of \cite{Gu}.

\begin{lemma}
\label{lemma:determinantal_exp_and_var}
Let $x$ be a determinantal point process with points lying entirely in an interval $J$, \footnote{$J$ may be $\mathbb{R}$.} with continuous kernel $K: \, J \times J \rightarrow \mathbb{R}$ representing an orthogonal projection on $L^2(J)$. For a subinterval $I \subseteq J$, let $X_I$ be the number of points of the process $x$ that lie in $I$. Then
$$
\mathbb{E}\, X_I = \int_I K(x,x)\, dx,
$$
$$
\Var\, X_I = \int_I \int_{J \setminus I} K(x,y)^2\, dxdy.
$$
\end{lemma}

We are now in a position to prove our lemma.
\begin{proof}[Proof of Proposition \ref{proposition:SO(2n)_EV_bounds1}]
We treat the bound for expectation first. Note that by applying Lemma \ref{lemma:determinantal_exp_and_var},
\begin{equation}
\label{W_exp}
\mathbb{E}\, W^{(2n)}_I = \int_I K_{SO(2n)}(x,x)\, dx = \int_I [ (2n-1)+ s_{2n-1}(2x)]\, dx.
\end{equation}
Yet recalling the definition of $s_{2n-1}(2x)$ and from inspection of this Dirichlet kernel, for any interval $I$,
$$
\bigg| \int_{I} \frac{\sin 2\pi (2n-1) x}{\sin 2\pi x}\, dx \bigg| \leq \int_{-\tfrac{1}{2(2n-1)}}^{\tfrac{1}{2(2n-1)}}  \frac{\sin 2\pi (2n-1) x}{\sin 2\pi x}\, dx \leq 1,
$$
so \eqref{W_exp} is
$$
(2n-1) |I| + O(1) = 2n|I|+O(1).
$$

Turning to variance, we have from Lemma \ref{lemma:determinantal_exp_and_var},
\begin{equation}
\label{var_bound_1}
\Var\, W^{(2n)}_I = \int_{[0,\frac{1}{2}]\setminus I} \int_I K_{SO(2n)}(x,y)^2\, dx dy \lesssim \int_{[0,\frac{1}{2}]\setminus I} \int_I s_{2n-1}(x-y)^2 + s_{2n-1}(x+y)^2\, dx dy,
\end{equation}
with the upper bound following straightforwardly from the definition of $K_{SO(2n)}$. For typographical reasons we let $N = 2n-1$. Recall the well-known bound for the Dirichlet kernel:
\begin{equation}
\label{Dirichlet_bound}
s_N(t) = O\Big(\frac{N}{1+ N \|t\|}\Big),
\end{equation}
where $\|t\|$ denotes the distance of $t$ to the nearest integer. Hence for $t \in (-\frac{1}{2},\frac{1}{2})$,
\begin{equation}
\label{Dirichlet_(-1/2,1/2)}
s_N(t)^2 \lesssim \frac{N^2}{1+ N^2 t^2},
\end{equation}
while for $t \in (0,1)$,
\begin{equation}
\label{Dirichlet_(0,1)}
s_N(t)^2 \lesssim \frac{N^2}{1+ N^2 t^2} + \frac{N^2}{1+N^2 (1-t)^2}.
\end{equation}
Hence from \eqref{Dirichlet_(-1/2,1/2)},
\begin{align}
\label{s_n_var_bound}
 \int_{[0,\frac{1}{2}]\setminus I} \int_I s_{N}(x-y)^2 \, dx dy &\leq  \int_{\mathbb{R}\setminus I} \int_I \frac{N^2}{1+N^2(x-y)^2} \, dx dy  \\ \notag &= \bigg( \int_{-\infty}^0 + \int_{|I|}^\infty\bigg) \bigg( \int_0^{|I|} \frac{N^2}{1+N^2(x-y)^2}\, dx\bigg) dy \\
\notag &= \bigg( \int_{-\infty}^0 + \int_{N|I|}^\infty\bigg) \bigg( \int_0^{N|I|} \frac{1}{1+(x-y)^2}\, dx\bigg) dy
\end{align}
with the second line following from translating $x$ and $y$ so that the starting point of the interval begins at the origin and the third line following from making a change of variable. Swapping the order of integration, with a short computation we bound the above by
$$
\lesssim \int_0^{N|I|} \frac{1}{1+x} + \frac{1}{1+ (N|I|-x)}\, dx \lesssim \log(2+ N|I|).
$$

On the other hand, using \eqref{Dirichlet_(0,1)},
\begin{multline*}
\int_{[0,\frac{1}{2}]\setminus I} \int_I s_{2n-1}(x+y)^2 \, dx dy \leq \int_{[0,\frac{1}{2}]} \int_I s_{2n-1}(x+y)^2 \, dx dy \\
\lesssim \int_{[0,\frac{1}{2}]} \int_I \frac{N^2}{1+N^2(x+y)^2} \, dx dy + \int_{[0,\frac{1}{2}]} \int_I \frac{N^2}{1+N^2(1-x-y)^2}\, dx dy.
\end{multline*}
Supposing that $I = [A,B]$, and making a change of variable,
\begin{align*}
\int_{[0,\frac{1}{2}]} \int_I \frac{N^2}{1+N^2(x+y)^2} \, dx dy &= \int_0^{N/2} \int_{NA}^{NB} \frac{1}{1+(x+y)^2}\, dx dy \\
&\lesssim \int_{NA}^{NB} \frac{1}{1+y}\, dy \lesssim \log\big(2+ N(B-A)\big) = \log(2+ N|I|).
\end{align*}
In the same way, making the change of variable $x \mapsto \frac{1}{2}-x$, $y \mapsto \frac{1}{2}-y$ and using the same computation,
$$
\int_{[0,\frac{1}{2}]} \int_I \frac{N^2}{1+N^2(1-x-y)^2}\, dx dy \lesssim \log(2+N|I|).
$$
Collecting these bounds and returning to \eqref{var_bound_1}, we see that 
$$
\Var\, W_I^{(2n)} \lesssim \log(2+N|I|),
$$
which verifies \eqref{SO(2n)_var_bound}.
\end{proof}

From this and symmetry it is easy to see a slightly more general result:
\begin{proposition}
\label{proposition:SO(2n)_EV_bounds2}
Let $I \subset [-\frac{1}{2},\frac{1}{2}]$ be an interval, and choose $g \in SO(2n)$ randomly according to Haar measure. For eigenangles $\vartheta_j$ labelled by \eqref{SO(2n)_eigenangles} as before, define the random variable $W_I^{(2n)} := \#\{ \vartheta_j \in I\}+ \#\{-\vartheta_j \in I\}$, which counts eigenangles in $I$. Then
$$
\mathbb{E}\, W_I^{(2n)} = 2n|I| + O(1),
$$
$$
\Var\, W_I^{(2n)} \lesssim \log(2+ 2n|I|),
$$
uniformly for all $n$ over all such intervals $I$.
\end{proposition}

With a virtually identical proof, using \eqref{SO(2n+1)_determinantal} and \eqref{SO(2n+1)_determinantal2}, we obtain the same estimate for eigenangles of $SO(2n+1)$. To summarize this information, we introduce one last labeling of the eigenangles of $g\in SO(n)$ (for $n$ even or odd). We write the eigenvales of $g$ as
\begin{equation}
\label{SO(n)_eigenangles}
\{e^{i2\pi \theta_1},...., e^{i2\pi \theta_n}\}
\end{equation}
with $\theta_j \in [-\frac{1}{2},\frac{1}{2})$ for all $j$. We have

\begin{proposition}
\label{proposition:SO(n)_EV_bounds}
Let $I \subset [-\frac{1}{2},\frac{1}{2})$ and choose $g \in SO(n)$ randomly according to Haar measure. For $\theta_j$ as in \eqref{SO(n)_eigenangles}, define the random variable $U_I^{(n)}:= \#\{ \theta_j \in I\}$, which counts the number of eigenangles in $I$. Then
$$
\mathbb{E}\, U_I^{(n)} = n|I| + O(1),
$$
$$
\Var\, U_I^{(n)} \lesssim \log(2+ n|I|).
$$
\end{proposition}

\subsection{The limiting characteristic polynomial \texorpdfstring{for $\xi_n^{SO}(s)$}{for the orthogonal group} }

We finally turn to a proof of Theorem \ref{thm:ortho}. We must recall another well known fact.

\begin{proposition}
\label{proposition_SO(n)_tends_to_sine}
Fix nonzero $E \in (-\frac{1}{2},\frac{1}{2})$. Define the point process $\tilde{x}_n$ by the point configurations
$$
\{\tilde{x}^{(n)}_j\}_{1\leq j \leq n} = \{n (\theta_j-E)\}_{1\leq j \leq n},
$$
with $\{e^{i2\pi \theta_1},...,e^{i 2\pi \theta_n} \}$ the eigenvalues of $g \in SO(n)$ chosen randomly according to Haar measure, with $\theta_j \in [-\frac{1}{2},\frac{1}{2})$ for all $j$. In law the point processes $\tilde{x}^{(n)}$ tends to the Sine process.
\end{proposition}

Note that the point process $\tilde{x}^{(n)}$ consist of the eigenangles of $g$ centered around $E$ and stretched out by a factor of $n$.

\begin{proof} 
For $E\in(0,\frac{1}{2})$ this follows from taking a limit of the explicit expression for correlation functions of $SO(2n)$ and $SO(2n+1)$, given in \eqref{SO(2n)_determinantal} and \eqref{SO(2n+1)_determinantal}. More precisely, the 
point process $\tilde{x}^{(n)}$ is determinantal and a direct computation implies that its kernel converges towards the sine-kernel, 
uniformly on compact sets. Now, since the kernels are associated to operators of orthogonal projections and then are "good kernels" (the eigenvalues of the corresponding operators are between $0$ and $1$), we can apply Lemma 4.2.48 of \cite{AGZ10} in order to conclude. The proof is similar for $E \in (-\frac{1}{2},0)$. 
\end{proof}

We finally can turn to a proof of the theorem at hand.

\begin{proof}[Proof of Theorem \ref{thm:ortho}] 
We note that for $\xi_n^{SO}(s)$ defined as in the theorem, and $\{e^{i2\pi \theta_1}, ..., e^{i2\pi \theta_n}\}$ the eigenvalues of $g \in SO(n)$ with $\theta_j \in [-\frac{1}{2},\frac{1}{2})$ for all $j$ as above,
\begin{align}
\label{SO_product}
\notag \xi_n^{SO}(s) =& \prod_{j=1}^n \frac{e^{i 2\pi(E+s/n)} - e^{i 2\pi \theta_j}}{e^{i 2\pi E} - e^{i 2\pi \theta_j}}, \\
\notag =& e^{i \pi s} \prod_{j=1}^n \frac{\sin(\pi( \theta_j - E - s/n))}{\sin(\pi(\theta_j - E))} \\
\notag =& e^{i \pi s} \lim_{V \rightarrow\infty} \prod_{j=1}^n \prod_{\nu = -V}^V \frac{\theta_j - E - s/n + \nu}{\theta_j - E + \nu} \\
=& e^{i \pi s} \lim_{B \rightarrow \infty} \prod_{|x_i^{(n)}| < B} \Big(1 - \frac{s}{x_i^{(n)}}\Big),
\end{align}
with the point processes $x^{(n)}$ defined by 
\begin{equation}
\label{x_SO(n)_def}
\{x_i^{(n)}\}_{i\in \mathbb{Z}} = \{ n(\theta_j - E + \nu)\}_{1 \leq j \leq n,\, \nu \in \mathbb{Z}}. 
\end{equation}
(This point process is the \emph{pull-back} of eigenangles of $g$ from $\mathbb{R}/\mathbb{Z}$ to $\mathbb{R}$, centered at $E$ and stretched out by a factor of $n$.)

We apply Theorem \ref{general_pp_theorem}. Clearly by applying Proposition \ref{proposition_SO(n)_tends_to_sine}, the process $x^{(n)}$ tends to the Sine process along with $\tilde{x}^{(n)}$. To prove Theorem \ref{thm:ortho} then we need only demonstrate that the processes $x^{(n)}$ are uniformly product-amenable. It is plain from the integer translations $\nu$ in the definition \eqref{x_SO(n)_def} that the condition \eqref{converging_sums} regarding the convergence of sums is true. Moreover, from examining Proposition \ref{proposition:SO(n)_EV_bounds}, it is clear that for the random variable $X_I^{(n)}:= \#\{x_i^{(n)} \in I\}$, where $I$ is any interval,
$$
\mathbb{E}\, X_I^{(n)} = |I| + O(1),
$$
and
$$
\Var X_I^{(n)} \lesssim \log(2+ |I|),
$$
(In fact, if $|I| \geq n$, then the variance will be substantially smaller than this.) This implies \emph{(i)} symmetry and \emph{(ii)} regularity in expectation, and likewise \emph{(iii)} regularity in variance of the processes $x^{(n)}$, and so the product representation \eqref{SO_product} implies the theorem.
\end{proof}

\subsection{The limiting characteristic polynomial \texorpdfstring{for $\xi_n^{Sp}(s)$}{for the symplectic group}}

A proof of Theorem \ref{thm:symp} is very similar. We outline the relevant facts:

Since the symplectic group is defined only for even orders, we may consider $g \in Sp(2n)$. The $2n$ eigenvalues may be labeled
\begin{equation}
\label{Sp_ev_label1}
\{e^{\pm i 2\pi \vartheta_1}, ..., e^{\pm i 2\pi \vartheta_n}\},
\end{equation}
with $\vartheta_j \in [0,\frac{1}{2}]$ for all $j$. For $g$ chosen randomly according to Haar measure, the points $\{\vartheta_1, ..., \vartheta_j\}$ form a determinantal point process: for $\eta \in C([0,\frac{1}{2}]^k)$,
\begin{equation}
\label{Sp(2n)_determinantal}
\mathbb{E} \sum_{\substack{j_1, ..., j_k \\  \mathrm{distinct}}} \eta(\vartheta_{j_1}, ..., \vartheta_{j_k}) = \int_{[0,\tfrac12]^k} \eta(x_1,...,x_k) \det_{1\leq i,j \leq k} \big[K_{Sp(2n)}(x_i,x_j)\big]\, d^k x,
\end{equation}
with
\begin{equation}
\label{Sp(2n)_determinantal2}
K_{Sp(2n)}(x,y):= s_{2n+1}(x-y) - s_{2n-1}(x+y),
\end{equation}
with $s_n(x) := \sin \pi n x / \sin \pi x$, as before. $K_{Sp(2n)}$ represents an orthogonal projection on $L^2[0,\frac{1}{2}]$ and we can prove analogues of Propositions \ref{proposition:SO(2n)_EV_bounds1} - \ref{proposition_SO(n)_tends_to_sine} in an identical fashion. We record the most important of these, reverting to the labeling
\begin{equation}
\label{Sp(n)_eigenangles}
\{ e^{i2\pi \theta_1},...,e^{i 2 \pi \theta_n}\}
\end{equation}
for all eigenvalues of $g \in Sp(n)$ (with $n$ even), with $\theta_j \in [-\frac{1}{2},\frac{1}{2})$. We have

\begin{proposition}
\label{proposition:Sp(n)_EV_bounds}
Let $n$ be even, $I \subset [-\frac{1}{2},\frac{1}{2})$, and choose $g \in Sp(n)$ randomly according to Haar measure. For $\theta_j$ as in \eqref{Sp(n)_eigenangles}, define the random variable $U_I^{(n)}:= \#\{ \theta_j \in I\}$, which counts the number of eigenangles in $I$. Then
$$
\mathbb{E}\, U_I^{(n)} = n|I| + O(1),
$$
$$
\Var\, U_I^{(n)} \lesssim \log(2+ n|I|).
$$
\end{proposition}

\begin{proposition}
\label{proposition_Sp(n)_tends_to_sine}
Fix nonzero $E \in (-\frac{1}{2},\frac{1}{2})$. Define the point process $\tilde{x}_n$ by the point configurations
$$
\{\tilde{x}^{(n)}_j\}_{1\leq j \leq n} = \{n (\theta_j-E)\}_{1\leq j \leq n},
$$
with $\theta_j$ as in \eqref{Sp(n)_eigenangles} the eigenvalues of $g \in Sp(n)$ chosen randomly according to Haar measure. In law the point processes $\tilde{x}^{(n)}$ tend to the Sine process.
\end{proposition}

The proof of Theorem \ref{thm:ortho} then applies almost word-for-word to prove Theorem \ref{thm:symp}.

\section{An application to the Gaussian Unitary Ensemble}

Finally we can turn to a proof of Theorem \ref{thm:main}, characterizing the limit of $\Xi_n^{GUE}(s)$. As we have mentioned, this provides an alternative approach to that outlined by Sodin \cite{So} making use of a result of Aizenman-Warzel \cite{AiWa}. Our approach will be similar to the last section; a new ingredient is a `localization'  of $\Xi_n^{GUE}(s)$. We approximate this ratio of determinants by a product, involving only eigenvalues close to $E\sqrt{n}$, multiplied by a regular term coming from the semicircular law. In addition to uniform bounds on counts of eigenvalues which do not seem to appear elsewhere in the literature, this localization may be of independent interest.

\subsection{Well-known facts about the GUE and an outline of our proof}

For $M$ a $n\times n$ GUE matrix, label the $n$ eigenvalues $\{\lambda_1,...,\lambda_n\}$. Asymptotically almost surely, all eigenvalues satisfy $|\lambda_i| \leq (2+o(1))\sqrt{n}$ as $n\rightarrow\infty$ and moreover,

\begin{thm}[The semicircular law]
\label{thm:semicircular}
For any continuous and compactly supported function $f :\mathbb{R} \rightarrow\mathbb{C}$,
$$
\lim_{n\rightarrow\infty} \E\left[ \frac{1}{n} \sum_{i=1}^n f\left( \frac{\lambda_i}{\sqrt{n}} \right) \right] = \int \rho_{sc}(x) f(x) \,dx,
$$
where $\lambda_1, \dots,\lambda_n$ are the eigenvalues of an $n\times n$ GUE matrix $M$ and $\rho_{sc}$ is defined in (\ref{sem_law}).
\end{thm}

Because of this result, it will be convenient to deal with eigenvalues at a \emph{macroscopic} scaling also. In the macroscopic scaling, eigenvalues have order $O(1)$, range from $-2$ to $2$ and are denoted $\Lambda_i$, with
$$
\Lambda_i := \frac{\lambda_i}{\sqrt{n}}.
$$

Owing to the semicircular law, eigenvalues $\Lambda_i$ lying in the interval $(-2,2)$ are said to lie \emph{in the bulk}. Inside the bulk, we will make use of a well-known microscopic scaling. It is also known that with a microscopic scaling, the eigenvalues of $M$ tend as $n\rightarrow\infty$ to the Sine process \cite[Ch. 3]{AGZ10}, \cite{GaMe}:

\begin{thm}[Gaudin-Mehta]
\label{thm:gaudin_mehta}
For any $E\in (-2,2)$ the point process with point configurations given by
$$
\Big\{ n \rho_{sc}(E)\Big(\Lambda_i- E\Big) \Big\}
$$
tends as $n\rightarrow\infty$ in law to the Sine process.
\end{thm}

For fixed $n$, we have moreover (a result also due to Gaudin-Mehta):

\begin{equation}
\label{GUE(n)_determinantal}
\mathbb{E} \sum_{\substack{j_1, ..., j_k \\  \mathrm{distinct}}} \eta(\lambda_{j_1}, ..., \lambda_{j_k}) = \int_{\mathbb{R}^k} \eta(x_1,...,x_k) \det_{1\leq i,j \leq k} \big[K_{GUE(n)}(x_i,x_j)\big]\, d^k x,
\end{equation}
with
\begin{equation}
\label{GUE(n)_determinantal2}
K_{GUE(n)}(x,y):= \sum_{k=0}^{n-1} \psi_n(x)  \psi_n(y) = \sqrt{n}\frac{\psi_n(x)\psi_{n-1}(y) - \psi_{n-1}(x)\psi_n(y)}{x-y},
\end{equation}
for $\psi_n(x)$ defined by
$$
\psi_n(x):= \frac{H_n(x)}{(2\pi)^{1/4} \sqrt{n!}} e^{-x^2/4},
$$
with $H_n(x)$ the `probabilists' Hermite polynomial of degree $n$, defined by
$$
H_n(x) := (-1)^n e^{x^2/2} \frac{d^n}{dx^n} e^{-x^2/2}
$$
and satisfying
$$
\int_\mathbb{R} \psi_i(x) \psi_j(x) \, dx = \delta_{ij}.
$$
(Hence $H_n$ are the orthogonal polynomials with leading coefficient $1$ for the measure $e^{-x^2/2}/\sqrt{2\pi}\, dx$.)

Theorem \ref{thm:gaudin_mehta} follows from \eqref{GUE(n)_determinantal} from an analysis of Placherel-Rotach estimates for Hermite polynomials (see \cite[Ch. 3]{AGZ10} for this deduction, along with the rest of the identities cited above). We will find it necessary to recall a few estimates of this sort later.

As in Theorems \ref{thm:ortho} and \ref{thm:symp}, it will be important to write $\Xi_n^{GUE}(s)$ in a form to which theorem \ref{general_pp_theorem} can be applied.

\begin{lemma}
\label{lemma:Xi_product}
$$ \Xi_n^{GUE}(s)
 = \prod_{i=1}^n \Big( 1 - \frac{s}
                                 {n \rho_{sc}(E)\left( \Lambda_i - E \right)} 
                 \Big)\\
$$
\end{lemma}
\begin{proof}
\begin{align*}
   \Xi_n^{GUE}(s)
& = \frac{\det\Big(E-\frac{M}{\sqrt{n}}+\frac{s}{n \rho_{sc}(E)}\Big)}{\det(E-\frac{M}{\sqrt{n}})}\\
& = \prod_{i=1}^n \left( \frac{E - \Lambda_i + \frac{s}{n \rho_{sc}(E)} }
                              {E - \Lambda_i } 
                  \right)\\
& = \prod_{i=1}^n \left( 1 - \frac{s}
                                  {n \rho_{sc}(E)\left( \Lambda_i - E \right)} 
                  \right)
\end{align*}
\end{proof}

We can now briefly outline our approach for Theorem \ref{thm:main}. In the next section we develop estimates akin to Propositions \ref{proposition:SO(n)_EV_bounds} and \ref{proposition:Sp(n)_EV_bounds} to estimate the expectation and variance of counts of eigenvalues lying in an interval located in the bulk. In section \ref{section:localization} following it, we apply this along with a large deviation estimate of Ben Arous and Guionnet to approximate the product in Lemma \ref{lemma:Xi_product} with a truncated product involving only eigenvalues $\Lambda_i$ very close to $E$, multiplied by a deterministic constant arising from the semicircular law. The product of eigenvalues, localized around $E$ may then be shown to converge to $\xi_\infty(s)$ by applying Theorem \ref{general_pp_theorem}.

\subsection{Eigenvalue counts for the GUE}

In this section we prove an analogue of Propositions \ref{proposition:SO(n)_EV_bounds} and \ref{proposition:Sp(n)_EV_bounds} for eigenvalues of GUE matrices. It will be most convenient in this setting to count points already at a microscopic scaling (as opposed to the macroscopic scaling we used in the statement of Propositions \ref{proposition:SO(n)_EV_bounds} and \ref{proposition:Sp(n)_EV_bounds} for $SO(n)$ and $Sp(n)$). Fix $E \in (-2,2)$ and $\delta > 0$ such that $[E-\delta, E + \delta] \subset (-2,2).$ For any interval $I = [a,b]$ define the count
\begin{align*}
X_I:=& \# \{ \Lambda_i: n \rho_{sc}(E)(\Lambda_i - E) \in I\} \\
=& \# \{ \lambda_i \in \sqrt{n}E + \frac{I}{\rho_{sc}(E)\sqrt{n}}\}.
\end{align*}

As before, we establish estimates on the expectation and variance of these counts uniformly in $I$ for all $I \subset [-\delta n , \delta n]$ (so that we will consider only $\Lambda_i \in [E-\frac{\delta}{\rho_{sc}(E)}, E+\frac{\delta}{\rho_{sc}(E)}]$ in our counts -- we will gain control of $\Lambda_i$ outside of this set later). A main tool is the following estimate,

\begin{lemma}
\label{K_GUE_bound}
For fixed $E \subset (-2,2)$ and $\delta > 0$ such that $[E - \delta, E + \delta] \subset (-2,2)$,
$$
K_{GUE(n)}(x,y) \lesssim n^{\frac{1}{2}} \wedge |x-y|^{-1},
$$
uniformly for $x,y \in [(E-\delta)\sqrt{n}\,, \,(E+\delta)\sqrt{n}]$ and all $n$. (The implied constant depends only upon $E$ and $\delta$.)
\end{lemma}

This lemma depends on an analysis of Plancherel-Rotach estimates, and we will prove it later. For the moment we take this estimate for granted, and show that it implies

\begin{proposition}
\label{proposition:point_count_estimates}
Fix $E \in (-2,2)$ and $\delta > 0$ such that $[E-\delta, E+\delta] \subset (-2,2)$. For an interval $I \subset [-\delta n, \delta n]$
$$
\Var\, X_I \lesssim \log(2+ |I|).
$$
Moreover for $I \subset [-n^{9/10},n^{9/10}]$
$$
\mathbb{E}\, X_I = |I| + O(|I|^{8/9}).
$$
\end{proposition}

\begin{proof}
We treat expectation first. Denoting $\rho_n(x):= K_{GUE(n)}(x \sqrt{n},x \sqrt{n})/\sqrt{n}$ the density
of the average empirical measure of the GUE at the macroscopic scale, we have
$$ \E\, X_I = n \int_{E + \frac{I}{n \rho_{sc}(E)} } \rho_n(x)\, dx
                        = \int_{I} \frac{ \rho_n(E+\frac{x}{n \rho_{sc}(E)}) }
                                        { \rho_{sc}(E) } dx.$$
Then the result is implied by the following (weaker form of) the uniform estimate found  in \cite{ErcolaniMcLaughlin03}, equation (4.2):
$$ \rho_n(x) = \rho_{sc}(x) + O( n^{-1} ).$$
For we have then
$$
\E\, X_I = \int_{I} \frac{ \rho_{sc}(E+\frac{x}{n \rho_{sc}(E)}) }{ \rho_{sc}(E) } dx + O\Big(\frac{|I|}{n}\Big) = |I| + O\Big(\frac{|I|^2}{n}+\frac{|I|}{n}\Big)
$$
using that $\rho_{sc}(E+\frac{x}{n \rho_{sc}(E)}) = \rho_{sc}(E) + O\left(\frac{x}{n\rho_{sc}(E)}\right)$ to estimate the integral. Because $|I| = O(n^{9/10})$ the claimed estimate follows.

For the variance, we again make use of Lemma \ref{lemma:determinantal_exp_and_var} to see that
$$
\Var\left( X_I \right) = \int_{x \in J} \int_{y \notin J} K_{GUE(n)}\left( x,y\right)^2 dy dx \ ,
$$
where we have abbreviated $J:=\sqrt{n}E + \frac{I}{\rho_{sc}(E) \sqrt{n}}.$ Let $W:= [(E-\delta)\sqrt{n}\,, \,(E+\delta)\sqrt{n}]$, and note
\begin{align}
\label{y_int_GUE_bound}
\int_{y \notin J} K_{GUE(n)}(x,y)^2\, dy & = \int_{y \in W \setminus J} K_{GUE(n)}(x,y)^2\, dy + \int_{y \notin W} K_{GUE(n)}(x,y)^2\, dy \\
\notag &\lesssim \int_{y \in W \setminus J} \frac{n}{1+ n(x-y)^2}\, dy + \int_{y \notin W} \psi_n^2(x) \psi_{n-1}^2(y) + \psi_n^2(y) \psi_{n-1}^2(x)\, dy
\end{align}
with the bound for the integral over $y \in W\setminus J$ following from Lemma \ref{K_GUE_bound}, and the bound for the integral over $y \notin W$ following from \eqref{GUE(n)_determinantal2} and the fact that for $y \notin W$ and $x \in J$, we have $|x-y| \gtrsim \sqrt{n}$. From the fact that $\psi_i$ are orthonormal,
$$
\int_{y \notin W} \psi_n^2(x) \psi_{n-1}^2(y) + \psi_n^2(y) \psi_{n-1}^2(x)\, dy \leq \psi_n^2(x) + \psi_{n-1}^2(x).
$$
But then integrating \eqref{y_int_GUE_bound} over $x \in J$,
\begin{align*}
\Var (X_I) &\lesssim \int_{x \in J} \int_{y \in W \setminus J} \frac{n}{1+ n(x-y)^2}\, dy dx  + \int_{x \in J} \psi_n^2(x) + \psi_{n-1}^2(x)\, dx\\
& \lesssim \int_{x \in J} \int_{y \in \mathbb{R} \setminus J} \frac{n}{1+ n(x-y)^2}\, dy dx + O(1).
\end{align*}
The integral here is of the same form we bounded in \eqref{s_n_var_bound}, and the same computation gives us the bound
$$
\Var (X_I) \lesssim \log(2 + \sqrt{n}|J|) \lesssim \log(2+|I|),
$$
as claimed.
\end{proof}

It remains to prove Lemma \ref{K_GUE_bound}.

\begin{proof}[Proof of Lemma \ref{K_GUE_bound}]
We recall that 
$$K_{GUE(n)}(x,y) = \sqrt{n} \frac{\psi_n(x)\psi_{n-1}(y) - 
\psi_n(y)\psi_{n-1}(x)}{x-y},$$
for $x \neq y$, and  
$$K_{GUE(n)}(x,x) = \sqrt{n} (\psi_n'(x) \psi_{n-1}(x) - \psi_n(x) \psi_{n-1}'(x) ).$$
where $(\psi_n)_{n \geq 1}$ are the Hermite functions.
For $x, y \in [(E-\delta) \sqrt{n}, 
(E+\delta) \sqrt{n}]$ (in the bulk)
we have $\psi_n$ and $\psi_{n-1}$ dominated by $n^{-1/4}$, for example by  Plancherel-Rotach estimates. Hence, 
$$K_{GUE(n)}(x,y) = O(|x-y|^{-1})$$
on the bulk. 
On the other hand, we have the relation: 
$$\psi'_n(x) = \frac{\sqrt{n}\psi_{n-1}(x) - \sqrt{n+1}\psi_{n+1}(x)}{2},$$
from classical recursion relations satisfied by the Hermite polynomials (see for example, 
\cite{AGZ10}, Lemma 3.2.7),
which implies that $\psi'_n$ is dominated by $n^{1/4}$ in the bulk. We deduce that 
$ K_{GUE(n)}(x,x)$ is dominated by $n^{\frac{1}{2}}$ in the bulk, and 
 by positivity of the kernel: 
$$|K_{GUE(n)}(x,y)| \leq [K_{GUE(n)}(x,x) K_{GUE(n)}(y,y)]^{\frac{1}{2}} = O (n^{\frac{1}{2}}).$$
\end{proof}

\subsection{Localization in the bulk}
\label{section:localization}

In this section we show that $\Xi_n^{GUE}(s)$ can be approximated in probability by a truncated product involving only those eigenvalues $\Lambda_i$ lying very close to $E$, multiplied by a deterministic constant coming from the semicircular law. 

We first localize the product giving us $\Xi_n^{GUE}(s)$ at the macroscopic level. Since the pioneering work of Ben Arous and Guionnet \cite{BenGui97}, it is understood that the semicircular law \ref{thm:semicircular} can be obtained as a corollary of a large deviation principle for the empirical spectral measure:
\begin{align}
\label{eq:def_esd}
\mu_n := & \frac{1}{n} \sum_{i=1}^n \delta_{\Lambda_i }.
\end{align}
See Hiai and Petz \cite{HiPe00} for a comprehensive survey.

\begin{thm}[LDP for empirical measure]
\label{thm:LDP}
On the Polish space of probability measures $\mathcal{M}_1 \left( \R \right)$ endowed with the L\'evy-Prokhorov topology, the empirical spectral measure $\mu_n$ satisfies a large deviation principle with speed $n^2$ and good rate function:
$$ I(\mu) := \iint \mu(dx) \mu(dy) \left( \half\left( x^2 + y^2 \right) - \log|x-y| \right)$$
\end{thm}

We make use of the identity demonstrated in Lemma \ref{lemma:Xi_product} in the argument that follows.

Let $\delta > 0$ small enough so that $E + [-\delta; \delta] \subset (-2,2)$. The eigenvalues $\left( \Lambda_i^{(n)}; 1 \leq i \leq n \right)$ can be separated into two regions: $[ E - \delta; E + \delta ]$ and the complement. Based on that distinction, we have:
\begin{align*}
\Xi_n^{GUE}(s)
& = \prod_{ |\Lambda_i - E| \leq \delta }
    \left( 1 - \frac{s}{n \rho_{sc}(E)\left( \Lambda_i - E \right)} \right)
    \prod_{ |\Lambda_i - E| > \delta }
    \left( 1 - \frac{s}{n \rho_{sc}(E)\left( \Lambda_i - E \right)} \right)
\end{align*}

We will show:
\begin{proposition}
\label{proposition:localization1}
For all $\varepsilon>0$ and compact sets $K$, for all sufficiently small $\delta > 0$ (depending on $\varepsilon, E$ and $K$), there exist constants $c,C$ (depending on $\varepsilon, E, K,$ and $\delta$) such that,
$$ \P\left( \sup_{s \in K} \left| 
   \prod_{ |\Lambda_i- E| > \delta }
    \left( 1 - \frac{s}{n \rho_{sc}(E)\left( \Lambda_i- E \right)} \right)
    -
    \exp\left(  
               - \frac{s}{\rho_{sc}(E)} \int dx \frac{\rho_{sc}(x)}{x - E}
        \right)
    \right| > \varepsilon \right) \leq C e^{-c n^2} \ .$$
The integral $\int dx \frac{\rho_{sc}(x)}{x - E}$ has to be understood in the sense of principal value.
\end{proposition}
\begin{proof}
We start by picking $\delta>0$ (depending on $\varepsilon$, $E$ and $K$) small enough so that uniformly in $s \in K$:
$$ \left| \frac{s}{\rho_{sc}(E)} \int dx \frac{\rho_{sc}(x)}{x - E} - 
          \frac{s}{\rho_{sc}(E)} \int_{|x|>\delta} dx \frac{\rho_{sc}(x)}{x - E} \right| \leq \frac{\varepsilon}{3} \ .$$
Then, 
we have that $|\Lambda_i - E| > \delta$ and $s \in K$ imply    
$$  \left| \log\left|1 - \frac{s}{n \rho_{sc}(E)\left( \Lambda_i - E \right)} \right|
    + \frac{s}{n \rho_{sc}(E)\left( \Lambda_i - E \right)} \right|
    \ll_{K,\delta,E} \left( \frac{1}{n \left( \Lambda_i- E \right)} \right)^2$$
for $n$ large enough depending on $K, \delta, E$. In this regime,
\begin{align*}
  & \prod_{ |\Lambda_i - E| > \delta }
    \left( 1 - \frac{s}{n \rho_{sc}(E)\left( \Lambda_i - E \right)} \right)\\
= & \exp\left(  
               - \frac{s}{n \rho_{sc}(E)} \sum_{ |\Lambda_i - E| > \delta } \frac{1}{\Lambda_i - E} + \frac{1}{n^2} O\left( \sum_{ |\Lambda_i - E| > \delta } \left( \frac{1}{\Lambda_i - E} \right)^2 \right)
        \right)\\
= & \exp\left(  
               - \frac{s}{\rho_{sc}(E)} \int_{|x-E| > \delta } \frac{\mu_n(dx)}{x - E} +
               O\left( \frac{1}{n \delta^2} \right)
        \right)
\end{align*}
Now, because of the large deviation principle (Theorem \ref{thm:LDP}), there are constants $c,C=c_{\varepsilon,E,K,\delta}, C_{\varepsilon,E,K,\delta}$ such that:
\begin{align*}
       \P\left( \sup_{s \in K} \left| \frac{s}{\rho_{sc}(E)} \int_{|x-E| > \delta } \frac{\mu_n(dx)}{x - E}
                     - \frac{s}{\rho_{sc}(E)} \int_{|x-E| > \delta } \frac{\rho_{sc}(dx)}{x - E} \right| > \frac{\varepsilon}{3} \right) \        
\leq & C e^{-c n^2} \ .
\end{align*}
Thanks to the choice of $\delta$, one can replace $\int_{|x-E| > \delta } \frac{\rho_{sc}(dx)}{x - E}$ by $\int \frac{\rho_{sc}(dx)}{x - E}$ at the cost of changing $\frac{\varepsilon}{3}$ to $\frac{2 \varepsilon}{3}$. Finally:
\begin{align*}
     &  \P\Bigg( \sup_{s\in K} \bigg| 
       \prod_{ |\Lambda_i - E| > \delta } \left( 1 - \frac{s}{n \rho_{sc}(E)\left( \Lambda_i - E \right)} \right)
        -  \exp\Big(-\frac{s}{\rho_{sc}(E)} \int \frac{\rho_{sc}(dx)}{x - E}\Big) \bigg| \\ & \qquad \qquad >  
       \exp\Big(-\frac{s}{\rho_{sc}(E)} \int \frac{\rho_{sc}(dx)}{x - E}\Big) (e^{2 \varepsilon/3
       + O\left( \frac{1}{n \delta^2} \right)} - 1) \Bigg)
\leq  C e^{-c n^2}, 
\end{align*}
by using the inequality
$$|e^\alpha - e^{\beta}| \leq e^{\beta} (e^{|\alpha - \beta|} - 1).$$
By suitably changing the value of  $\varepsilon$ (depending on $K, \delta, E$), we get the conclusion of the proposition
if $n$ is large enough, depending on $K, \varepsilon, \delta, E$. Then, small values of $n$ can be absorbed in the constant $C$.
\end{proof}

Here we can write more simply, for $E \in (-2,2)$,
$$
\int dx\, \frac{\rho_{sc}(x)}{x-E} = -\frac{E}{2},
$$
which is an exercise in calculus (see \cite[Ch. 4]{Meh}). Thus as a consequence of the above proposition,

\begin{corollary}
\label{corollary:localization2}
For any $\varepsilon > 0$ and any compact set $K$, for all sufficiently small  $\delta > 0$ (depending on $\varepsilon, E$ and $K$), we have
$$
\lim_{n\rightarrow\infty} \mathbb{P}\Big( \sup_{s\in K}\bigg|\,  \Pi_{\delta}^{-1}  \Xi_n^{GUE}(s) - e^{\half s E/\rho_{sc}(E)}  \, \bigg| \geq \varepsilon  \Big) = 0
$$
where 
$$\Pi_{\delta} := 
\prod_{|\Lambda_i - E|\leq \delta} \Big( 1 - \frac{s}{n \rho_{sc}(E)(\Lambda_i - E)}\Big).$$
\end{corollary}
We now bootstrap this estimate:

\begin{corollary}
\label{corollary:localization3}
For any $\varepsilon > 0$ and any compact set $K$, we have
$$
\lim_{n\rightarrow\infty} \mathbb{P}\Big( \sup_{s\in K}  \bigg|\, \Pi_{n^{-1/10}}^{-1}\Xi_n^{GUE}(s) - e^{\half s E/\rho_{sc}(E)}  \bigg| \geq \varepsilon  \Big) = 0,
$$
where
$$\Pi_{n^{-1/10}} :=  \prod_{|\Lambda_i - E|\leq n^{-1/10}} \Big( 1 - \frac{s}{n \rho_{sc}(E)(\Lambda_i - E)}\Big).$$
\end{corollary}
\begin{proof}
We observe that for all $B > 0$, 
 \begin{align*}
 & \mathbb{P}\Big( \sup_{s\in K}   | ( \Pi_{n^{-1/10}}^{-1} -
 \Pi_{\delta}^{-1}) \Xi_n^{GUE}(s) | \geq \varepsilon  \Big)
 =   \mathbb{P}\left( \sup_{s\in K}   \left| \frac{\Pi_{\delta}}{\Pi_{n^{-1/10}}} - 1 \right|
|\Pi_{\delta}^{-1} \Xi_n^{GUE}(s)|  \geq \varepsilon  \right)
\\ & \leq  \mathbb{P}\left( \sup_{s\in K}   \left| \frac{\Pi_{\delta}}{\Pi_{n^{-1/10}}} - 1 \right|
  \geq \varepsilon/B  \right)
  +  \mathbb{P}\left( \sup_{s\in K}  
|\Pi_{\delta}^{-1} \Xi_n^{GUE}(s)|  \geq B  \right).
 \end{align*} 
 From Corollary \ref{corollary:localization2}, the second term tends to zero when $n$ goes to infinity, 
 if $B$ is large enough depending on $K$ and $E$. By suitably changing $\varepsilon$ (depending on $K$ and $E$) and applying Corollary \ref{corollary:localization2} again, we easily deduce 
 that it is enough to  show the following: for any $E \in (-2,2)$, $\varepsilon > 0$ and compact $K$, there exists $\delta > 0$ arbitrarily small such that
\begin{equation}
\label{meso_localize}
\lim_{n\rightarrow\infty}\mathbb{P} \left[ \sup_{s\in K} \left|
\prod_{n^{-1/10} <  |\Lambda_i - E|\leq \delta } \left( 1 - \frac{s}{n \rho_{sc}(E) (\Lambda_i - E)} \right) - 1 \right| > \varepsilon 
\right] = 0.
\end{equation}
We will be able to demonstrate this if  we  show for all fixed $\delta > 0$ that
$$
\int_{n^{-1/10} < |x-E| \leq \delta} 
\frac{\mu_n(dx)}{x-E} \underset{n \rightarrow \infty}{\longrightarrow} 
\int_{-\delta}^{\delta} 
\frac{\rho_{sc} (x)}{x-E}  dx
$$
in probability, since this implies for sufficiently small $\delta$,
$$
\lim_{n\rightarrow\infty}\mathbb{P} \left[ \left| \int_{n^{-1/10} < |x-E| \leq \delta} 
\frac{\mu_n(dx)}{x-E} \right| > \varepsilon 
\right] = 0,
$$
which can be applied to control the product in \eqref{meso_localize} in the same fashion as the proof of Proposition \ref{proposition:localization1} above.

Yet since $\rho_n (x):= K_{GUE(n)} (x\sqrt{n}, x \sqrt{n}) /\sqrt{n}
= \rho_{sc}(x) + O(n^{-1})$, for any positive $\delta$
$$\int_{-\delta}^{\delta} 
\frac{\rho_{sc} (x)}{x-E}  dx - 
\int_{n^{-1/10} < |x - E| \leq \delta}
\frac{\rho_{n} (x)}{x-E}  dx$$
 tends to zero when $n$ goes to infinity
 and it is enough to show that 
 $$\int_{n^{-1/10} < |x-E| \leq \delta} 
\frac{\tilde{\mu}_n(dx)}{x-E} \underset{n \rightarrow \infty}{\longrightarrow} 
0$$
in probability, where $\tilde{\mu}_n (dx)
 = \mu_n (dx) - \rho_n(x) dx$ is the "compensated" empirical measure of the eigenvalues. 

We have, for $\delta$ small enough, $a, b \in [E-\delta, E + \delta]$, $a <b $
$$\mathbb{E} [|\tilde{\mu}_n ([a,b])|] 
\leq  \sqrt {\operatorname{Var} 
 (\mu_n([a,b]))} = O \left( n^{-1} \sqrt{\log n} \right),$$
by using Proposition \ref{proposition:point_count_estimates}.
By using integration by parts, we have
$$\int_{n^{-1/10} < x-E \leq \delta} 
\frac{\tilde{\mu}_n(dx)}{x-E}
= \left[ \frac{\tilde{\mu}_n [E,x]}{x-E}
 \right]_{(n^{-1/10}, \delta]}
 + \int_{n^{-1/10} < x-E \leq \delta}
 \frac{\tilde{\mu}_n [E,x]}{(x-E)^2} dx,
 $$
 which easily implies
 $$\mathbb{E} \left[\left|\int_{n^{-1/10} < x-E \leq \delta} 
\frac{\tilde{\mu}_n(dx)}{x-E} \right| \right] \underset{n \rightarrow \infty}{\longrightarrow} 0.$$
Similarly, 
 $$\mathbb{E} \left[\left|\int_{n^{-1/10} < - ( x-E )\leq \delta} 
\frac{\tilde{\mu}_n(dx)}{x-E} \right| \right] \underset{n \rightarrow \infty}{\longrightarrow} 0,$$
which gives the desired result. 
\end{proof}

\subsection{The limiting characteristic polynomial \texorpdfstring{for $\Xi_n^{GUE}(s)$}{for GUE} }

Theorem \ref{thm:main} is now a simple deduction based on the general Theorem \ref{general_pp_theorem}. We define point processes $x^{(n)}$ with point configurations given by
$$
\{x^{(n)}_i\} = \{ n \rho_{sc}(E)(\Lambda_i-E): |\Lambda_i-E| < n^{-1/10}\}.
$$
Note that each point configuration consists of only finitely many points, and moreover is defined in terms of only those eigenvalues that are localized around the macroscopic value $E$. Our motivation for this definition is that Corollary \ref{corollary:localization3} tells us that on compact sets $\Xi_n^{GUE}(s)$ is approximated in probability by
$$
e^{sE/2\rho_{sc}(E)} \prod_i \Big( 1- \frac{s}{x_i^{(n)}}\Big).
$$

\begin{proof}[Proof of Theorem \ref{thm:main}]
It is clear from Proposition \ref{proposition:point_count_estimates} that the sequence of processes $x^{(1)}, x^{(2)}, ...$ are uniformly product-amenable. Moreover Theorem \ref{thm:gaudin_mehta} (Gaudin-Mehta) tells us that the processes $x^{(n)}$ tend to the Sine process. Hence as $n\rightarrow\infty$
$$
\prod_i \Big( 1- \frac{s}{x_i^{(n)}}\Big) \rightarrow e^{-i \pi s} \xi_\infty(s),
$$
in law in the topology of uniform convergence on compact sets, by Theorem \ref{general_pp_theorem}. Corollary \ref{corollary:localization3} thus implies
$$
\Xi_n^{GUE}(s) \rightarrow e^{sE/2\rho_{sc}(E) - i \pi s} \xi_\infty(s),
$$
as claimed.
\end{proof}

\section{Acknowledgments}

We would like to thank Elizabeth Meckes for an informative response regarding some of the bounds proved in section \ref{classicalcompact}, Sasha Sodin likewise for a helpful discussion, and an anonymous referee for several useful comments and corrections. B.R. was partially supported during this research by the NSF grant DMS-1701577.

\bibliographystyle{halpha}
\bibliography{Bib_InvRatios}

\end{document}